\documentclass[preprint,12pt]{elsarticle}

\usepackage{amssymb}
\usepackage{array}
\usepackage{mathtools}
\usepackage{graphicx}

\usepackage{amsthm}
\newtheorem{theorem}{Theorem}[section]
\newtheorem{corollary}{Corollary}[section]
\newtheorem{lemma}{Lemma}[section]
\newtheorem{definition}{Definition}[section]

\journal{Journal of Number Theory}

\begin{document}

\begin{frontmatter}



\title{A step towards proving de Polignac's Conjecture}


\author{J. Sellers}


\begin{abstract}
Consider the set of all natural numbers that are co-prime to  primes less than or equal to a given prime. Then given a consecutive pair of numbers in that set with an arbitrary even gap, we prove there exists an unbounded number of actual prime pairs with that same gap. This conditional proof of de Polignac's conjecture constitutes a proof for a range of known gaps, but the full conjecture requires additional proof that such number pairs exist for all even gaps.
\end{abstract}


%
%
%

%

\end{frontmatter}


\section{Introduction}
French mathematician Alphonse de Polignac conjectured in 1849 that: "Every even number is the difference of two  consecutive primes in infinitely many ways."\cite{dP,LD}  The subsumed twin prime conjecture is more well known and is considered older, but its origin is not otherwise documented. de Polignac's conjecture, a generalization for arbitrary even gaps, is taken as the earliest documented statement that is inclusive of the twin prime conjecture. Work on prime gaps has application to both de Polignac's ocnjecture and the twin prime conjecture, but the twin prime conjecture appears to have been the primarily goal of most work.
 
Maynard in \cite{JM1} gives an excellent overview of approaches to the twin prime conjecture. The earliest result comes in the work of Hardy and Littlewood \cite{HL} where they proposed a prime pair counting function using a modified assumption about the Riemann Hypothesis to characterize the density of prime pairs. 

Sieve theory has made the most significant recent progress. Originally proposed by Brun \cite{VB} as a modified form of the sieve of Eratosthenes and applied to the Goldbach Conjecture. His significant result proved that the sum of the reciprocal of twin primes converges. Sieve theory was further developed by Selberg \cite{AS} and has made significant advances applying the work of Bombieri, Friedlander, and Iwaneic \cite{BFI1,BFI2,BFI3} on the distribution of primes in arithmetic progression and then applying the results of Goldston, Pintz, and Yildririm \cite{GPY} on primes in tuples. This culminated in the work of Zhang \cite{YZ} who combined these approaches and proved the existence of a finite, though very large limit on gaps, for which there are infinite prime pairs. His method was subsequently modified to significantly reduce the gap limit, to 246,.\cite{JM2, poly1, poly2}. 

Those latter approaches formulated sieves using a product of linear functions chosen to ensure finding at least two prime numbers in an infinite number of tuples of fixed finite size. Therefore, while it has produced significant progress, it does not demonstrate a result for prime pairs of a specific gap and is known to have inherent limitations for reducing the gap limit further.

The primary difference in this paper is that we work in the realm of relative primes rather than attempting to deal with primes directly, because relative primes are more easily predicted. The set of numbers prime to $P\le P_{k}$ includes the set of all prime numbers greater than $P_{k}$ and all composite numbers whose prime factors are all greater than $P_{k}$. All of these fall in the the two arithmetic progressions $6n+5$ and $6n+7$. All such relative primes between the composite numbers are actual prime numbers. The difficulty in predicting prime numbers derives from the inability to order composite numbers beyond $P_{k+1}^2<P_{k+1}P_{k+2}$ without knowing their actual values. However, we do know that all numbers less than $P_{k+1}^2$ that are prime to $P<P_{k}$ are actual prime numbers. In that domain our results are applicable to actual prime numbers.

The various combinations of prime factors $P\le P_{k}$ repeat identically in successive sequences of $P_{k}\#$ numbers. Using this, we define prospective primes, numbers prime to $P\le P_{k}$ for some $P_{k}$, among which all prime numbers geater than $P_{k}$ must occur.
We then apply a formulaic approach for the specification of prospective primes in successively larger sets of $P_{k}\# \rightarrow P_{k+1}\#$ numbers. We see that gaps between consecutive prospective primes propagate predictably between successively larger sets, whereas gaps between actual primes do not. This allows us to assess their distribution directly and prove they exist in a range where they must also be actual prime pairs of a given gap.

This work represents an extension of \cite{JS} which addressed only twin primes, extending it to gaps of arbitrary even numbers. In this approach there are two parts to proving de Polignac's conjecture.  Part one, shown in this work, proves that given any consecutive prospective prime pair of even gap $g$, there exists an unbounded number of actual prime pairs with gap $g$. The second part, partially addressed in this work, requires one to prove there exists a pair of consecutive prospective primes for any arbitrary even gap. We show that such gaps exist between consecutive prospective prime pairs for $g=P_{k}\pm 1$ and $g=P_{k+1}-P_{k}$ for all $P_{k}$, however to complete the proof of de Polignac's conjecture one must show that such gaps exist for all even numbers.

\section{Definitions and framework}

$P$ = generic prime number 

$P_{k} = k^{th}$ prime number $(P_{1}=2)$

$P_{k}\#= \prod_{i=1}^{k}P_{i}$

$S_{k}:=\left\{N:  5\le N \le 4+ P_{k}\#\right\}; \; N\in \mathbb{N}$	

$S_{k}^{(m)}:=\left\{N: 5+mP_{k-1}\#\le N \le 4+(m+1)P_{k-1}\#\right\}$; where:

\[
0\le m\le P_{k}-1; \quad S_{k}^{(m)}\subset S_{k};\quad S_{k}^{(0)}=S_{k-1}
\]
\[
\cup_{m=0}^{P_{k}-1}S_{k}^{(m)}=S_{k}
\quad \& \quad
S_{k}^{(m)}\cap S_{k}^{(m')}=
\begin{cases}
	\emptyset &\textrm{if} \quad m\ne m'\\
	S_{k}^{(m)} &\textrm{if}\quad  m= m'
\end{cases}
\]

$\widetilde{P}_{\{k\}}=$ unspecified prospective prime number in $S_{k}$:	
	\[
 \forall{P} \left[P|\widetilde{P}_{\{k\}}\longrightarrow  P> P_{k} \right]
  \]
$\qquad\widetilde{P} =$  generic prospective prime ; prime to all $P\le P_{l}$ for unspecified $P_{l}$

$\:\;\widetilde{\mathbb{P}}_{k}:=\left\{\widetilde{P}_{\{k\}}\in S_{k}\right\}$, the set of all prospective primes in $S_{k}$

$\:\;\widetilde{\mathbb{P}}_{k}^{(m)}:=\left\{\widetilde{P}_{\{k\}}\in S_{k}^{(m)}\right\}$, the set of all prospective primes in subset $S_{k}^{(m)}$

$(\widetilde{PgP})=$ generic prospective prime pair with gap $g$

$(\widetilde{PgP})_{k}=$ generic prospective prime pair with gap $g$ in $S_{k}$

\begin{definition}
	Two prospective prime numbers, $\widetilde{P}_{\{k\}} < \widetilde{P}_{\{k\}}'$  are considered consecutive prospective prime numbers, when there is no prospective prime number between them, i.e.:
	\[
	\forall{N} \left[ \left(\widetilde{P}_{\{k\}}<N<\widetilde{P}_{\{k\}}'\right)\longrightarrow \left(P|N\rightarrow P\le P_{k}\right)\right]
	\]
\end{definition}

When we refer to prospective prime pairs we always mean consecutive prospective prime pairs.

Prospective prime numbers, prime to all $P\le P_{k}$ have the form:

\begin{equation}\label{E:genproprime}
	\widetilde{P}_{\{k\}}=\left(\begin{array}{c} 5 \\ 7 \end{array}\right)+\sum_{j=3}^{k}m_{j}P_{j-1}\#
\end{equation}

For $\widetilde{P}_{\{k\}}\in S_{k}$, $m_{k}$ is constrained by: $0\le m_{k} \le P_{k}-1$. In addition two values of $m_{j}$ for each $j$, corresponding separately to the 5 and 7 in (\ref{E:genproprime}) are disallowed to avoid a result divisible by $P_{j}$.\footnote{If we allow all values $m_{j} \ge 0$, then 
(\ref{E:genproprime}) represents the progressions $6n+5$ and $6n+7$.} This is best handled iteratively as in the following:

Going from $S_{k}\rightarrow S_{k+1}$ we get:

\begin{equation}\label{E: nextproprime}
	\widetilde{P}_{\{k+1\}}=\widetilde{P}_{\{k\}}+m_{k+1}P_{k}\# \qquad 0 \le m_{k+1} \le P_{k+1}-1
\end{equation}

$\widetilde{P}_{\{k+1\}}$ remains prime to $P\le P_{k}$ and will be prime to $P_{k+1}$ as long as we insist $P_{k+1} \nmid \widetilde{P}_{\{k+1\}}$, enforced by $m_{k+1}\ne \widehat{m}_{k+1}$, where:\footnote{(\ref{E: defmhat}) follows from (\ref{E: nextproprime}) letting $\widetilde{P}_{\{k+1\}}\bmod{P_{k+1}}=0$}.

\begin{equation}\label{E: defmhat}
	\widehat{m}_{k+1}=\frac{\alpha P_{k+1}-\widetilde{P}_{\{k\}}\bmod{P_{k+1}}}{\left(P_{k}\#\right)\bmod{P_{k+1}}}
\end{equation}
and where $\alpha$ is the smallest integer such that $\widehat{m}_{k+1}$ is an integer $\le P_{k+1}-1$. Also,
\[
\widetilde{P}_{\{k\}}\bmod{P_{k+1}}=0  \longleftrightarrow \alpha =0
\].
One can see from (\ref{E: defmhat}) that the values of $\widehat{m}_{k+1}$ are distinct for $\widetilde{P}_{\{k\}}$ belonging to distinct residue classes $\bmod{P_{k+1}}$ and all $\widetilde{P}_{\{k\}}$ in the same residue class $\bmod{P_{k+1}}$ have the same value for $\widehat{m}_{k+1}$.

Note that:

\begin{equation}\label{E: proprimeinsub}
\widetilde{P}_{\{k+1\}}=\widetilde{P}_{\{k\}}+m_{k+1}P_{k}\# \in  S_{k+1}^{(m_{k+1})}
\end{equation}

Therefore each prospective prime number in $S_{k}$ generates one prospective prime number in all but one subset of $S_{k+1}$. The one disallowed subset being $S_{k+1}^{(\widehat{m}_{k+1})}$.

It follows from (\ref{E: proprimeinsub}) that for $m'>m$ and if $\widetilde{P}_{\{k\}}\in S_{k}^{(m)}$ and $\widetilde{P}_{\{k\}}'\in S_{k}^{(m')}$ then $\widetilde{P}_{\{k\}}<\widetilde{P}_{\{k\}}'$. Therefore, consecutive prospective primes can only occur within a subset or between the largest prospective prime in one subset and the least prospective prime in the next sequential subset.

It is also important to know that prospective primes using (\ref{E:genproprime}) are unique in accordance with the following lemma.

\begin{lemma}\label{L: uniqueness}
Given 
\[
\widetilde{P}_{\{k\}}=\left(\begin{array}{c} 5\\7  \end{array}\right)+\sum_{j=3}^{k}m_{j}P_{j-1}\#
\]
 and
\[
\widetilde{P}_{\{k\}}'=\left(\begin{array}{c} 5\\7  \end{array}\right)+\sum_{j=3}^{k}m_{j}'P_{j-1}\#
\]
where $0\le m_{j},m_{j}' \le P_{j}-1$.
Then, 
\[
\widetilde{P}_{\{k\}}=\widetilde{P}_{\{k\}}' \longleftrightarrow m_{j}=m_{j}' \quad \textrm{for}\quad 3\le j \le k \quad \textrm{and both either start with 5 or both with 7}
\]

\end{lemma}
\begin{proof}
Taking: $\widetilde{P}_{\{k\}}'=\widetilde{P}_{\{k\}}$ gives:
\[
\sum_{j=3}^{k}(\pm\Delta m_{j})P_{j-1}\#=\left(\begin{array}{c} 0\\2  \end{array}\right)
\]
where the zero applies if $\widetilde{P}_{k}$ and $\widetilde{P}_{k}'$ both start with 5 or both start with 7, and 2 applies if one starts with 5 and the other starts with 7. 

The smallest finite value for the left hand side of the equation is $6$. Therefore it cannot be solved by finite integral values of $\Delta m_{j}$ and the only solution is $\sum_{j=3}^{k}(\pm\Delta m_{j})P_{j-1}\#=0$, where $\Delta m_{j}=0$ for all $j$.
\end{proof}

\section{Prospective Prime pairs with gap $g$}

We call prospective prime numbers, prime to all $P\le P_{k}$, consecutive if there are no numbers prime to all $P\le P_{k}$ between them.\footnote{Consecutive prime numbers may be taken as consecutive prospective prime numbers, but only if there are no prospective prime numbers between them. }
Gaps between consecutive prospective prime pairs both propagate unchanged and are increased when generating prospective numbers via  (\ref{E: nextproprime}). Increases occur due to the supplemental condition $m_{k+1}\ne \widehat{m}_{k+1}$. For example, let $\widetilde{P}_{\{k\}}<\widetilde{P}_{\{k\}}'< \widetilde{P}_{\{k\}}''$ be three consecutive prospective prime numbers in $S_{k}$, with gaps $g=\widetilde{P}_{\{k\}}'-\widetilde{P}_{\{k\}}$ and $g'=\widetilde{P}_{\{k\}}''-\widetilde{P}_{\{k\}}'$. Then Equation~(\ref{E: nextproprime}) gives the following numbers in $S_{k+1}$ which remain prime to $P\le P_{k}$:

\[
\widetilde{P}_{\{k+1\}}=\widetilde{P}_{\{k\}}+m_{k+1}P_{k}\#
\]
\[
\widetilde{P}_{\{k+1\}}'=\widetilde{P}_{\{k\}}'+m_{k+1}'P_{k}\#
\]
\[
\widetilde{P}_{\{k+1\}}''=\widetilde{P}_{\{k\}}''+m_{k+1}''P_{k}\#
\]

In cases where $m_{k+1}=m_{k+1}'=m_{k+1}''$ the gaps remain at $g$ and $g'$. However, we must consider the disallowed cases given by the supplemental condition (\ref{E: defmhat}), which is necessary so that the corresponding numbers in $S_{k+1}$ are prime to $P\le P_{k+1}$.

Note that $\widehat{m}_{k+1}$, $\widehat{m}_{k+1}'$, and $\widehat{m}_{k+1}''$ are distinct from each other unless $g\bmod{P_{k+1}}=0$, $g'\bmod{P_{k+1}}=0$, or $(g+g')\bmod{P_{k+1}}=0$. Then given that there are $P_{k+1}-1$ valid values for each,  there are the following cases when $\widehat{m}_{k+1}$, $\widehat{m}_{k+1}'$, and $\widehat{m}_{k+1}''$ are distinct:\footnote{$g^{?}$ represents an unspecified gap, which is the gap from the disallowed prospective prime to the next larger or smaller prospective prime, respectively.}\\

\begin{enumerate}
	\item \underline{$m_{k+1}=m_{k+1}'=m_{k+1}'' \ne \widehat{m}_{k+1},\widehat{m}_{k+1}', \widehat{m}_{k+1}''$:}  Yeilds $P_{k+1}-3$ cases where both gaps are preserved, because all three of the corresponding prospective primes are allowed in those corresponding subsets:
	
	\[
	\widetilde{P}_{\{k\}}\leftarrow g \rightarrow \widetilde{P}_{\{k\}}'\leftarrow g' \rightarrow \widetilde{P}_{\{k\}}'' \quad\Longrightarrow\quad \widetilde{P}_{\{k+1\}}\leftarrow g \rightarrow \widetilde{P}_{\{k+1\}}'\leftarrow g' \rightarrow \widetilde{P}_{\{k+1\}}''
	\]
	
	\item \underline{$m_{k+1}=m_{k+1}'=m_{k+1}'' = \widehat{m}_{k+1}$:}  Yeilds 1 case where only the second gap is preserved, because $\widetilde{P}_{\{k+1\}}$ is disallowed in $S_{k+1}^{(\widehat{m}_{k+1})}$:
	
	\[
	\widetilde{P}_{\{k\}}\leftarrow g \rightarrow \widetilde{P}_{\{k\}}'\leftarrow g' \rightarrow \widetilde{P}_{\{k\}}''\quad \Longrightarrow \quad \leftarrow g^{?} + g \rightarrow \widetilde{P}_{\{k+1\}}'\leftarrow g' \rightarrow \widetilde{P}_{\{k+1\}}''
	\]

	\item \underline{$m_{k+1}=m_{k+1}'=m_{k+1}'' = \widehat{m}_{k+1}'$:}  Yeilds 1 case where the two gaps merge, because $\widetilde{P}_{\{k+1\}}'$ is disallowed in $S_{k+1}^{(\widehat{m}_{k+1}')}$:
	
	\[
	\widetilde{P}_{\{k\}}\leftarrow g \rightarrow \widetilde{P}_{\{k\}}'\leftarrow g' \rightarrow \widetilde{P}_{\{k\}}'' \quad\Longrightarrow\quad \widetilde{P}_{\{k+1\}}\leftarrow g + g' \rightarrow \widetilde{P}_{\{k+1\}}''
	\]
	
	\item \underline{$m_{k+1}=m_{k+1}'=m_{k+1}'' = \widehat{m}_{k+1}''$:}  Yeilds 1 case where only the first gap is preserved, because $\widetilde{P}_{\{k+1\}}''$ is disallowed in $S_{k+1}^{(\widehat{m}_{k+1}'')}$:
	
	\[
	\widetilde{P}_{\{k\}}\leftarrow g \rightarrow \widetilde{P}_{\{k\}}'\leftarrow g' \rightarrow \widetilde{P}_{\{k\}}'' \quad\Longrightarrow\quad \widetilde{P}_{\{k+1\}}\leftarrow g \rightarrow \widetilde{P}_{\{k+1\}}'\leftarrow g' + g^{?} \rightarrow 
	\]
\end{enumerate}

One can see from this that if $\widehat{m}_{k}$, $\widehat{m}_{k}'$, $\widehat{m}_{k}''$ are not distinct, then case 1 would have $P_{k+1}-2$ cases if any two are equal and the third is distinct and would have $P_{k+1}-1$ cases if all three were equal.

Another important point from this example is why it is necessary to track prospective prime numbers rather than actual prime numbers. Consider the case in the above example where $\widetilde{P}_{\{j\}}=P_{\{j\}}$ and $\widetilde{P}_{\{j\}}''=P_{\{j\}}''$ are actual consecutive prime numbers. It is possible then that either one or both of $\widetilde{P}_{\{j+1\}}$ and $\widetilde{P}_{\{j+1\}}''$ may not be prime. If they are both prime it is possible that $\widetilde{P}_{\{j+1\}}'$ may also be prime. In these cases the gaps are not propagated unchanged and $\widetilde{P}_{\{j+1\}}$ and $\widetilde{P}_{\{j+1\}}''$ are not consecutive prime numbers. However, in the case of consecutive prospective prime numbers there are always predictable cases where the gaps are preserved and the prospective prime numbers remain consecutive. This is independent of whether the prospective prime numbers are prime or not. Consider, for example, the consecutive prime numbers in $S_{4}$, $113$ and $127$. While they are consecutive primes, they are not consecutive prospective primes because $121=11^2$ between them is a prospective prime in $S_{4}$, i.e., prime to $P\le 7$. Table~\ref{T: example1} shows how these three numbers propogate into $S_{5}$ along with their associated gaps.

\begin{table}[h]
\caption{
The table shows the propogation of consecutive prime numbers 113 and 127 from $S_{4}$ into $S_{5}$. The gap between prime numbers is only preserved in $S_{5}$ in cases where the intermediate prospective prime, 121, does not generate an actual prime and where the corresponding prospective primes generated by 113 and 127 are actual primes.
}\label{T: example1}
\begin{tabular}{|p{9mm}||p{6mm}|p{6mm}|p{6mm}|p{6mm}|p{6mm}|p{8mm}|p{8mm}|p{8mm}|p{8mm}|p{8mm}|p{8mm}|}
\multicolumn{12}{c}{$\widetilde{\mathbb{P}}_{5}^{(m)}=\widetilde{\mathbb{P}}_{4} +m\cdot 210 \qquad \textbf{bold} = \neg P  \qquad   \widehat{m}=\neg \widetilde{P}$}\\ 
\hline
m=	& \; 0 & \; 1 & \; 2 & \; 3 &\;  4 & \: 5 & \: 6 & \: 7 & \: 8 & \: 9 & \; 10 \\
	\hline
113	&113&\textbf{323} & 533 & 743 & 953 & 1163 & 1373 & 1583 & $\; \widehat{m}$ &  2003& 2213 \\
	\hline
$\:\; g$	&  & \;8 & \;8 & \;8 & \; 8 & \;8 & \;8 & \; 8 & & \; 8 & \: 8  \\
	\hline
121	&$\; \widehat{m}$ & 331 & 541 & 751 & \textbf{961} & 1171 & 1381 &  \textbf{1591}& 1801 & 2011 & 2221 \\
	\hline
$\:\; g'$	&  & \; 6 & \; 6 & \; 6 &\; 6  &  & \; 6&\; 6  & \; 6 & \; 6 & \; 6 \\
	\hline
127	&127  & 337 & 547 & 757 & 967 & $\; \widehat{m}$& \textbf{1387} &  1597& \textbf{1807} & 2017 & \textbf{2227} \\
	\hline
$g+g'$	& 14 &  &  &  & 14 &  &  & 14 &  &  &  \\
	\hline
\end{tabular}
\end{table}
\newpage
 The lesson here is that determining whether a prospective prime is an actual prime in a given subset is not as straightforward as predicting whether a prospective prime is present or disallowed in that subset as determined by $\widehat{m}$.

\subsection{Propagation of prospective prime pairs with gap $g$}

\begin{theorem}\label{T: noprotpp} Given set $S_{l}=\left\{N: 5 \le N \le 4+P_{l}\#\right\}$ containing a pair of consecutive prospective prime numbers, $(\widetilde{P}_{\{l\}},\widetilde{P}_{\{l\}}')$ with gap $\widetilde{P}_{\{l\}}'-\widetilde{P}_{\{l\}}=g$ and given any prime number, $P_{k} > P_{l}$, let $\mathring{n}_{k}^{g}$ be the number of prospective prime pairs $\left(\widetilde{P}_{\{k\}}, \widetilde{P}_{\{k\}}'\right)$ with gap $g$ in $S_{k}=\left\{N: 5 \le N \le 4+P_{k}\#\right\}$ that are derived from that prospective prime pair with gap $g$ in $S_{l}$, then 
	\[
	\mathring{n}_{k}^{g} = \prod_{i=l+1}^{k}(P_{i}-2)\cdot \prod_{\substack{i=l+1\\ P_{i}|g }}^{k}\frac{(P_{i}-1)}{(P_{i}-2)}
	\]
\end{theorem}

\begin{proof}
	
	Given a consecutive prospective prime pair with gap $g$ in $S_{j}$, $\left(\widetilde{P}_{\{j\}}\widetilde{P}_{\{j\}}'\right)$, we can define prospective prime pairs with gap $g$ in $S_{j+1}$ by:
	
	\begin{align}\label{E: defproprima-1} 
		\widetilde{P}_{\{j+1\}}=\widetilde{P}_{\left\{j\right\}}+m_{j+1}P_{j}\#\\
		\widetilde{P}_{\{j+1\}}'=\widetilde{P}_{\left\{j\right\}}'+m_{j+1}P_{j}\#\notag
	\end{align}
	with the supplementary conditions: $0\le m_{j+1} \le P_{j+1}-1$, $m_{j+1}\ne \widehat{m}_{j+1}$ and $m_{j+1}\ne \widehat{m}_{j+1}'$ where:
	
	\begin{align}\label{E: suppconda}\notag
		\widehat{m}_{j+1}&=\frac{\alpha P_{j+1}-\widetilde{P}_{\left\{j\right\}}\bmod{P_{j+1}}}{\left(P_{j}\#\right)\bmod{P_{j+1}}}
		\\
		\\ \notag
		\widehat{m}_{j+1}'&=\frac{\alpha' P_{j+1}-\widetilde{P}_{\left\{j\right\}}'\bmod{P_{j+1}}}{\left(P_{j}\#\right)\bmod{P_{j+1}}}
	\end{align}
	
	Given $m_{j+1}<P_{j+1}-1$ one can see that both $\widetilde{P}_{\{j+1\}}$ and $\widetilde{P}_{\{j+1\}}'$ are prime to all $P\le P_{j}$. Then the other supplementary condition guarantees that $\widetilde{P}_{\{j+1\}}$ and $\widetilde{P}_{\{j+1\}}'$ are both prime to $P_{j+1}$ and therefore they are a prospective prime pair with gap $g$ in $S_{j+1}$.
	
	Given $\widetilde{P}_{\{j\}}'=\widetilde{P}_{\{j\}}+g$, (\ref{E: suppconda}) gives:
	
	\begin{align}\label{E: disalloweddiff}\notag
		\widehat{m}_{j+1}'&=\frac{\alpha' P_{j+1}-(\widetilde{P}_{\left\{j\right\}}+g)\bmod{P_{j+1}}}{\left(P_{j}\#\right)\bmod{P_{j+1}}} \\ 
		&=\widehat{m}_{j+1}
		+ \frac{\Delta\alpha \cdot P_{j+1}-g\bmod{P_{j+1}}}{\left(P_{j}\#\right)\bmod{P_{j+1}}}
	\end{align}
	where $\Delta \alpha$ is modified from $\alpha'-\alpha$ to account for separating out $g$ in the mod function, and is chosen as the least integer to make the second term an integer.
	
	Consider the case where $g\bmod{P_{j+1}}=0$, then:
	\[
	\widetilde{P}_{\{j\}}'\bmod{P_{j+1}}=(
	\widetilde{P}_{\{j\}}+g)\bmod{P_{j+1}}=\widetilde{P}_{\{j\}}\bmod{P_{j+1}}
	\]
	In that case, there is only one disallowed subset in $S_{j+1}$, so $(P_{\{j\}},P_{\{j\}}')$ generates $P_{j+1}-1$ prospective prime pairs with gap $g$ in $S_{j+1}$. If $g\bmod{P_{j+1}}\ne 0$ then
	$m_{j+1}$ has $P_{j+1}-2$ allowed values and the prime pair $(P_{\{j\}},P_{\{j\}}')$ generates $P_{j+1}-2$ distinct prospective prime pairs with gap $g$ in $S_{j+1}$. 
	
	By the same procedure, those prospective prime pairs in $S_{j+1}$ each generate prospective prime pairs with gap $g$ in $S_{j+2}$:
	
	\begin{align}\label{E: defproprima-2} 
		\widetilde{P}_{\{j+2\}}=\widetilde{P}_{\{j+1\}}+m_{j+2}P_{j+1}\#\\
		\widetilde{P}_{\{j+2\}}'=\widetilde{P}_{\{j+1\}}'+m_{j+2}P_{j+1}\#\notag
	\end{align}
	with the supplementary conditions: $0\le m_{j+2} \le P_{j+2}-1$, $m_{j+2}\ne \widehat{m}_{j+2}$ and $m_{j+2}\ne \widehat{m}_{j+2}'$ where:
	
	\begin{align}\label{E: suppconda-2}\notag
		\widehat{m}_{j+2}&=\frac{\alpha_{j+2} P_{j+2}-\widetilde{P}_{\{j+1\}}\bmod{P_{j+2}}}{\left(P_{j+1}\#\right)\bmod{P_{j+2}}}
		\\
		\\ \notag
		\widehat{m}_{j+2}'&=\frac{\alpha_{j+2}' P_{j+2}-\widetilde{P}_{\{j+1\}}'\bmod{P_{j+2}}}{\left(P_{j+1}\#\right)\bmod{P_{j+2}}}
	\end{align}
	
	Again, $\widehat{m}_{j+2}$ and $\widehat{m}_{j+2}'$ are distinct unless $g\bmod{P_{j+2}}=0$ in which case the corresponding prospective prime pair in $S_{j+1}$ generates $P_{j+2}-1$ instead of $P_{j+2}-2$ prospective primes with gap $g$ in $S_{j+2}$.
	
Furthermore we know from Lemma~\ref{L: uniqueness} that the prospective primes generated in this process are distinct so that the prime pairs are also distinct pairs.
	
Then following the same process, successively generating prospective prime pairs of gap $g$, in larger sets, e.g. going from $S_{j}$ to $S_{j+1}$, each prospective prime pair with gap $g$ in $S_{j}$ generates $P_{j+1}-1$ distinct prospective prime pairs of gap $g$ in $S_{j+1}$ if $P_{j+1}$ is a factor in $g$ and otherwise generates $P_{j+1}-2$ distinct prospective prime pairs of gap $g$ in $S_{j+1}$. 
	
	Therefore in going from $S_{l}$ to $S_{k}$ the number of prospective prime pairs with gap $g$ in $S_{k}$ that are generated from a prospective prime pair with gap $g$ in $S_{l}$ is given by $ \prod_{i=l+1}^{k}(P_{i}-2)\cdot \prod_{\substack{i=l+1\\ P_{i}|g }}^{k}\frac{(P_{i}-1)}{(P_{i}-2)}$. Therefore we have: 
	
	\[
	\mathring{n}_{k}^{g} =\prod_{i=l+1}^{k}(P_{i}-2)\cdot \prod_{\substack{i=l+1\\ P_{i}|g }}^{k}\frac{(P_{i}-1)}{(P_{i}-2)}
	\]
	
\end{proof}

Assuming there exists set $S_{l}$ that contains at least one prospective prime pair with gap $g$, if that set contains $n_{l}^{g}$ such prospective prime pairs, then  the actual number of prospective prime pairs with gap $g$ in $S_{k}$, $k>l$, derived from those $n_{l}^{g}$ prospective prime pairs is:

\begin{equation}\label{E: totalprimepairs}
	n_{k}^{g}\ge n_{l}^{g}\cdot \mathring{n}_{k}^{g}	
\end{equation}

The equal sign holds if $g=2$, because prospective twin primes can all be generated from the single twin prime $(5,7)\in S_{2}$ using (\ref{E: nextproprime}) and (\ref{E: defmhat}), giving: \cite{JS}

\[
n_{k}^{2}=\prod_{i=3}^{k}(P_{i}-2)
\] 

The formulas in Theorem~\ref{T: noprotpp} and (\ref{E: totalprimepairs}) will generally represent a minimum when considering the total prospective prime pairs with gap $g>2$ in a set. This occurs because new larger gaps are always generated in going to a larger set because of the supplemental condition (\ref{E: defmhat}).

\subsection{Distribution of Prospective Prime pairs with gap $g$}\label{S: distribution}.

We define $(\widetilde{PgP})_{j}=\left( \widetilde{P}_{\{j\}}, \widetilde{P}_{\{j\}}' \right)$ as a generic prospective prime pair with gap $g$ in $S_{j}$

In the following Lemmas we assume there exists a set $S_{l}$ with at least one prospective prime pair with gap $g$. In the Lemmas, the indices $j$ and $k$ are assumed to have values $>l+2$.

\begin{lemma}\label{L: distrofSj+1}
	The set of $(\widetilde{PgP})_{j+1}\in S_{j+1}$ generated from a single $(\widetilde{PgP})_{j}\in S_{j}$ has each $(\widetilde{PgP})_{j+1}$ distributed to a distinct subset of $S_{j+1}$. Furthermore, if $g\bmod{P_{j+1}}=0$ they are distributed one each to all but one subset of $S_{j+1}$ and if $g\bmod{P_{j+1}}\ne 0$ they are distributed one each to all but two subsets of $S_{j+1}$.
\end{lemma}
\begin{proof}
	Let $\widetilde{(PgP)}_{j+1}=\left( \widetilde{P}_{\{j+1\}}, \widetilde{P}_{\{j+1\}}' \right)$ be a prospective prime pair with gap $g$ in $S_{j+1}$ generated from $(\widetilde{PgP})_{j}$, where:
	
	\begin{equation}\label{E: PgP 1}
		(\widetilde{PgP})_{j+1}=(\widetilde{PgP})_{j}+m_{j+1}P_{j}\#
	\end{equation}
	This actually represents separate equations relating $\widetilde{P}_{\{j+1\}}$ to $\widetilde{P}_{\{j\}}$ and $\widetilde{P}_{\{j+1\}}'$ to $\widetilde{P}_{\{j\}}'$ both using the same value of $m_{j+1}$, where:
	
	\[
	0 \le m_{j+1} \le P_{j+1}-1 
	\]
	and where additionally:
	
	\begin{align}\label{E: PgPms}\notag
		m_{j+1}\ne& \widehat{m}_{j+1} =\frac{\alpha_{j+1} P_{j+1}-\widetilde{P}_{\{j\}}\bmod{P_{j+1}}}{\left(P_{j}\#\right)\bmod{P_{j+1}}}\\ 
		\textrm{and}&\\ \notag
		m_{j+1}\ne& \widehat{m}_{j+1}'= \frac{\alpha_{j+1}' P_{j+1}-\widetilde{P}_{\{j\}}'\bmod{P_{j+1}}}{\left(P_{j}\#\right)\bmod{P_{j+1}}}
	\end{align}
	where $\alpha_{j+1}$ and $\alpha_{j+1}'$ represent the lowest integer values yielding integer solutions for $\widehat{m}_{j+1}$ and $\widehat{m}_{j+1}'$. 
	
	Given subsets of $S_{j+1}$:
	
	\begin{equation}\label{E: subsetdef}
		S_{j+1}^{(m)}=\left\{N: 5+mP_{j}\#\le N\le 4+(m+1)P_{j}\#\right\}	
	\end{equation}
	one can see that:
	
	\begin{equation}\label{E: PgP 2}
		(\widetilde{PgP})_{j+1}=(\widetilde{PgP})_{j}+m_{j+1}P_{j}\#\in S_{j+1}^{(m_{j+1})}
	\end{equation}
	where $0\le m_{j+1}\le P_{j+1}-1$.

	Therefore a fixed $(\widetilde{PgP})_{j}\in S_{j}$ generates one prospective prime pair with gap $g$ into each allowed subset of $S_{j+1}$. The disallowed subsets of $S_{j+1}$ are given by  (\ref{E: PgPms}) and are $S_{j+1}^{(\widehat{m}_{j+1})}$ and $S_{j+1}^{(\widehat{m}_{j+1}')}$. These will be the same single disallowed subset if $g\bmod{P_{j+1}}=0$, because then
	\[\widetilde{P}_{\{j\}}'\bmod{P_{j+1}}=\left(\widetilde{P}_{\{j\}}+g\right)\bmod{P_{j+1}}=\widetilde{P}_{\{j\}}\bmod{P_{j+1}}
	\]
	. Therefore, each $(\widetilde{PgP})_{j}\in S_{j}$ generates one  corresponding $(\widetilde{PgP})_{j+1}$ into all but one or two of the $P_{j+1}$ subsets of $S_{j+1}$  respectively, depending on whether $g\bmod{P_{j+1}}=0$ or not.
\end{proof}

\begin{lemma}\label{L: distroinSj+2}
	Given the set of $(\widetilde{PgP})_{j+2}\in S_{j+2}$  generated by a single  $(\widetilde{PgP})_{j}\in S_{j}$, then the disallowed subsets $S_{j+2}^{(\widehat{m})}$ corresponding to the two comoponents of each $(\widetilde{PgP})_{j+2}$ are separately distinct. 
\end{lemma}
\begin{proof}
	Consider the set of $(\widetilde{PgP})_{j+1}\in S_{j+1}$ generated from the same  $(\widetilde{PgP})_{j}\in S_{j}$, which we represent as: $\left\{(\widetilde{PgP})_{j+1}\right\}_{(\widetilde{PgP})_{j}}$.  The $(\widetilde{PgP})_{j+1} \in \left\{(\widetilde{PgP})_{j+1}\right\}_{(\widetilde{PgP})_{j}}$ are distributed in $S_{j+1}$ as given by Lemma~\ref{L: distrofSj+1}, one each to all but one or two subsets of $S_{j+1}$.
	
	Now consider the set of $(\widetilde{PgP})_{j+2}$ generated by the set of $\left\{(\widetilde{PgP})_{j+1}\right\}_{(\widetilde{PgP})_{j}}$. We represent this set as:
	
	\begin{equation}\label{E: PGPinJ+2}
		\left\{(\widetilde{PgP})_{j+2}\right\}_{(\widetilde{PgP})_{j}}= \left\{(\widetilde{PgP})_{j+1}\right\}_{(\widetilde{PgP})_{j}}+ m_{j+2}P_{j+1}\#
	\end{equation}
	where we consider that the second term on the right is added to both components of each member of the set represented as the first term on the right. We have supplementary conditions:
	
	\[
	0\le m_{j+2} \le P_{j+2}-1 \le \quad \textrm{and}\quad m_{j+2}\ne \widehat{m}_{j+2}, \widehat{m}_{j+2}'
	\]
	where, given $(\widetilde{PgP})_{j+1}=( \widetilde{P}_{\{j+1\}}, \widetilde{P}_{\{j+1\}}')$:
	
	\begin{align}\label{E: suppconda-2b}\notag
		\widehat{m}_{j+2}&=\frac{\alpha_{j+2} P_{j+2}-\widetilde{P}_{\{j+1\}}\bmod{P_{j+2}}}{\left(P_{j+1}\#\right)\bmod{P_{j+2}}}
		\\
		\\ \notag
		\widehat{m}_{j+2}'&=\frac{\alpha_{j+2}' P_{j+2}-\widetilde{P}_{\{j+1\}}'\bmod{P_{j+2}}}{\left(P_{j+1}\#\right)\bmod{P_{j+2}}}
	\end{align}
	
	These represent two distinct disallowed subsets in $S_{j+2}$ unless $g\bmod{P_{j+2}}=0$ in which case there is only one disallowed subset.
	
	By definition, each $\widetilde{P}_{j+1}\in \left\{(\widetilde{PgP})_{j+1}\right\}_{(\widetilde{PgP})_{j}}$ is generated using the same $(\widetilde{PgP})_{j}$.  Therefore, from Equations~(\ref{E: suppconda-2b}) we have:
	
	\begin{align}\label{E: suppconda-2c}\notag
		\widehat{m}_{j+2}&=\frac{\beta_{j+2} P_{j+2}-\widetilde{P}_{\{j\}}-m_{j+1}\left(P_{j}\#\right)\bmod{P_{j+2}}}{\left(P_{j+1}\#\right)\bmod{P_{j+2}}}
		\\
		\\ \notag
		\widehat{m}_{j+2}'&=\frac{\beta_{j+2}' P_{j+2}-\widetilde{P}_{\{j\}}-g\bmod{P_{j+2}}-m_{j+1}\left(P_{j}\#\right)\bmod{P_{j+2}}}{\left(P_{j+1}\#\right)\bmod{P_{j+2}}}
	\end{align}
	
	Where we use $\beta$ instead of $\alpha$ to represent possible changes to the integer values given the breakout of the mod arguments. However they still are the lowest integer values making  $\widehat{m}_{j+2}$ and $\widehat{m}_{j+2}'$  integers.
	
	One can see that for a given $\widetilde{P}_{\{j\}}$ and fixed $P_{j+2}$ the only variable in each of the equations in (\ref{E: suppconda-2c}) is $m_{j+1}$. According to Lemma~\ref{L: distrofSj+1} each $(\widetilde{PgP})_{j+1}\in \left\{(\widetilde{PgP})_{j+1}\right\}_{(\widetilde{PgP})_{j}}$ has a unique corresponding value of $m_{j+1}$, and therefore the values of $\widehat{m}_{j+2}$ and $\widehat{m}_{j+2}'$ are separately distinct corresponding to the values of $m_{j+1}$.
	Therefore the disallowed subsets for each component of 
	\[(\widetilde{PgP})_{j+2} \in \left\{(\widetilde{PgP})_{j+2}\right\}_{(\widetilde{PgP})_{j}}
	\] 
	namely $S_{j+2}^{(\widehat{m}_{j+2})}$ and $S_{j+2}^{(\widehat{m}_{j+2}')}$ are separately distinct.
\end{proof}

\begin{lemma}\label{L: delamhat}
	The separation of disallowed subsets corresponding to the two components of each $(\widetilde{PgP})_{k}\in S_{k}$ is a constant in $S_{k}$.	
\end{lemma}	
\begin{proof}
	Using (\ref{E: suppconda-2b}) with $k=j+2$	and $\widetilde{P}_{\{k-1\}}'=\widetilde{P}_{\{k-1\}}+g$ gives:
	
	\begin{equation}\label{E: fixeddelta}
		\Delta \widehat{m}_{k}=\frac{\Delta\alpha P_{k}-g\bmod{P_{k}}}{\left(P_{k-1}\#\right)\bmod{P_{k}}}
	\end{equation}
\end{proof}
Where all quantities on the right hand side of (\ref{E: fixeddelta}) are fixed given $S_{k}$.

\begin{lemma}\label{L: distrofPgPinSj_2}
	Given the set	$\left\{(\widetilde{PgP})_{j+2}\right\}_{(\widetilde{PgP})_{j}}$  of $(\widetilde{PgP})_{j+2}\in S_{j+2}$ generated by a single $(\widetilde{PgP})_{j}\in S_{j}$, each subset $S_{j+2}^{(m)}$ contains a minimum of  $P_{j+1}-4$ of the $(\widetilde{PgP})_{j+2}\in \left\{(\widetilde{PgP})_{j+2}\right\}_{(\widetilde{PgP})_{j}}$.
\end{lemma}
\begin{proof}
	Restating (\ref{E: PGPinJ+2}):
	\[
	\left\{(\widetilde{PgP})_{j+2}\right\}_{(\widetilde{PgP})_{j}}= \left\{(\widetilde{PgP})_{j+1}\right\}_{(\widetilde{PgP})_{j}}+ m_{j+2}P_{j+1}\#
	\]
	Lemma~\ref{L: distrofSj+1} gives that the 
	$(\widetilde{PgP})_{j+1}\in 
	\left\{(\widetilde{PgP})_{j+1}\right\}_{(\widetilde{PgP})_{j}}
	$
	are distributed one to a subset across all but one or two subsets of $S_{j+1}$. That means there are at least $P_{j+1}-2$ distinct $(\widetilde{PgP})_{j+1}\in \left\{(\widetilde{PgP})_{j+1}\right\}_{(\widetilde{PgP})_{j}}$.
	
	Applying Lemma~\ref{L: distrofSj+1} individually to each $(\widetilde{PgP})_{j+1}$ says that the corresponding $(\widetilde{PgP})_{j+2}$ are distributed one per subset across all but one or two subsets of $S_{j+2}$. This is true for each of the $P_{k+1}-2$ instances of $(\widetilde{PgP})_{j+1}$.
	
	Lemma~\ref{L: distroinSj+2} says that the disallowed subsets of $S_{j+2}$ are separately distinct for the lesser and greater components of the resulting $(\widetilde{PgP})_{j+2}$. Therefore none of the $(\widetilde{PgP})_{j+2}$ have the same disallowed subset corresponding to their lesser components and the same for their greater components. 
	
	It is possible however for the disallowed subsets of $S_{j+2}$ to be the same for the opposite components of two $(\widetilde{PgP})_{j+2} \in \left\{(\widetilde{PgP})_{j+2}\right\}_{(\widetilde{PgP})_{j}}$. This can occur when:
	
	\[
	(\widetilde{PgP})_{j+1}=(\widetilde{PgP})_{j+1}' \pm (nP_{j+2} + g)
	\]
	This can only occur if $g\bmod{P_{j+2}}\ne 0$; i.e., where the corresponding $(\widetilde{PgP})_{j+1}$ has two disallowed subsets when generating prospective prime pairs in $S_{j+2}$. 
	
	This means that a subset of $S_{j+2}$ can have at most two exclusions of $(\widetilde{PgP})_{j+2}$ and therefore there are at least $P_{j+1}-4$ of the $(\widetilde{PgP})_{j+2}$
	in each subset of $S_{j+2}$
\end{proof}

With these results we have the following theorem.

\begin{theorem}\label{T: distroPgP}
	Given the set $S_{l}=\left\{N: 5 \le N \le 4+P_{l}\#  \right\}$  containing at least one prospective prime pair with gap $g$. Then for $k>l+2$, consider the set 
	$S_{k}=\left\{N: 5 \le N \le 4+P_{k}\#\right\}$ with its $P_{k}$  subsets: 
	\[S_{k}^{(m)}=\left\{N: 5+mP_{k-1}\# \le N \le 4+(m+1)P_{k-1}\#\right\}
	\]
	$0 \le m \le P_{k}-1$.  
	Then if $\mathring{n}_{S_{k}^{(m)}}^{g}$ is the number of prospective prime pairs with gap $g$  in each subset $S_{k}^{(m)}\in S_{k}$ generated from a prospective prime pair with gap $g$ in $S_{l}$, then:
	\[
	\mathring{n}_{S_{k}^{(m)}}^{g}\ge \mathring{n}_{k-2}^{g}(P_{k-1}-4) = (P_{k-1}-4)\prod_{i=l+1}^{k-2}(P_{i}-2)\cdot \prod_{\substack{i=l+1\\ P_{i}|g }}^{k-2}\frac{(P_{i}-1)}{(P_{i}-2)}
	\]	
\end{theorem}

\begin{proof}
	Given Lemma~\ref{L: distrofPgPinSj_2} we know that for each $(\widetilde{PgP})_{k-2} \in S_{k-2}$ we have a minimum of $P_{k-1}-4$ prospective prime pairs with gap $g$ in each of the subsets $S_{k}^{(m)}$. Then using Theorem~\ref{T: noprotpp} we know there are $\mathring{n}_{k-2}^{g}=\prod_{i=l+1}^{k-2}(P_{i}-2)\cdot \prod_{\substack{i=l+1\\ P_{i}|g }}^{k-2}\frac{(P_{i}-1)}{(P_{i}-2)}$ prospective prime pairs in $S_{k-2}$ 
	
	Putting these two results together we get:
	
	\begin{align}\label{E: primespairsinsub}\notag
		\mathring{n}_{S_{k}^{(m)}}^{g} &\ge n_{k-2}^{g} \cdot (P_{k-1}-4) \\
		&= (P_{k-1}-4) \prod_{i=l+1}^{k-2}(P_{i}-2)\cdot \prod_{\substack{i=l+1\\ P_{i}|g }}^{k-2}\frac{(P_{i}-1)}{(P_{i}-2)}  \\ \notag
	\end{align}	
\end{proof}

\begin{corollary}
	For sufficiently large $P_{k}$:
	\[
	\mathring{n}_{S_{k}^{(m)}}^{g}\ge\mathring{n}_{k-1}^{g}-2\mathring{n}_{k-2}^{g}
	\]	
\end{corollary}

\begin{proof}
	We can also write the inequality~(\ref{E: primespairsinsub}) as:
	
	\begin{align}\label{E: thmversion} \notag
		\mathring{n}_{S_{k}^{(m)}}^{g} &\ge (P_{k-1}-4) \prod_{i=l+1}^{k-2}(P_{i}-2)\cdot \prod_{\substack{i=l+1\\ P_{i}|g }}^{k-2}\frac{(P_{i}-1)}{(P_{i}-2)} \\ \notag
		&=[(P_{k-1}-2)-2] \prod_{i=l+1}^{k-2}(P_{i}-2)\cdot \prod_{\substack{i=l+1\\ P_{i}|g }}^{k-2}\frac{(P_{i}-1)}{(P_{i}-2)} \\ \notag
		&=\left[\prod_{i=l+1}^{k-1}(P_{i}-2)-2\prod_{i=l+1}^{k-2}(P_{i}-2)\right]\cdot \prod_{\substack{i=l+1\\ P_{i}|g }}^{k-2}\frac{(P_{i}-1)}{(P_{i}-2)}\\ 
		&=
		\begin{cases} 
			\mathring{n}_{k-1}^{g}-2\mathring{n}_{k-2}^{g}&\textrm{if}\quad	P_{k-1} \nmid g \\
			\frac{(P_{k-1}-2)}{(P_{k-1}-1)}\mathring{n}_{k-1}^{g}-2\mathring{n}_{k-2}^{g} &\textrm{if}\quad	P_{k-1}|g
		\end{cases}
	\end{align}
	
	Note that by choosing $P_{k}$ sufficiently large, e.g., $P_{k}>P_{\boldsymbol{\pi}\left(P_{l}\#\right)}>g$, only the first case in (\ref{E: thmversion}) applies.
\end{proof}

\begin{corollary}\label{C: asymtoticlimit}
Given the minimum distribution of $(\widetilde{PgP})_{k}$ across the subsets of $S_{k}$ as in Theorem~\ref{T: distroPgP}, that minimum assymtotically approaches the average distribution of $(\widetilde{PgP})_{k}$ to subsets of $S_{k}$:
	\[
\min{(\mathring{n}_{S_{k}^{(m)}}^{g})}\longrightarrow\frac{ \mathring{n}_{k}^{g}}{P_{k}}\quad \textrm{as} \quad k\longrightarrow \infty
	\]
\end{corollary}
\begin{proof}
	Given that $S_{k}$ has $P_{k}$ subsets, $S_{k}^{(m)}$, the stated minimum number of prospective prime pairs in each subset generated for each $(\widetilde{PgP})_{l}$ accounts for
	
	\[
	P_{k} \cdot(P_{k-1}-4)\prod_{i=l+1}^{k-2}(P_{i}-2)\cdot \prod_{\substack{i=l+1\\ P_{i}|g }}^{k-2}\frac{(P_{i}-1)}{(P_{i}-2)}
	\]
	of the $\mathring{n}_{k}^{g}$ total prospective prime pairs in $S_{k}$ for each $(\widetilde{PgP})_{l}$. Therefore the fraction of the total represented by the minimum is:
	
	\begin{align}\notag
		\frac{P_{k} \cdot \min{(\mathring{n}_{S_{k}^{(m)}}^{g})}}{\mathring{n}_{k}^{g}}
		=& \frac{P_{k}(P_{k-1}-4)\prod_{i=l+1}^{k-2}(P_{i}-2)\cdot \prod_{\substack{i=l+1\\ P_{i}|g }}^{k-2}\frac{(P_{i}-1)}{(P_{i}-2)}}{\prod_{i=l+1}^{k}(P_{i}-2)\cdot \prod_{\substack{i=l+1\\ P_{i}|g }}^{k-2}\frac{(P_{i}-1)}{(P_{i}-2)}}\\ \notag
		=&
		\frac{P_{k}(P_{k-1}-4)}{(P_{k}-2)(P_{k-1}-2)}\quad\textrm{then letting}\quad \Delta=P_{k}-P_{k-1}\\ \notag
		&=\frac{1-\frac{\Delta +2}{P_{k}-2}}{1-\frac{\Delta +2}{P_{k}}}<1
	\end{align}
Therefore the ratio, which is less than $1$ approaches $1$ as $k$ gets large, proving the corollary.
\end{proof}

Corollary~\ref{C: asymtoticlimit} means that when we consider the distribution of prospective prime pairs with gap $g$ in $S_{k}$ that there is no systematic allotment of more prospective prime pairs to one or a few subsets and overall the difference in allotments averages out.  Therefore we can say that prospective twin primes are fairly evenly distributed between the subsets of $S_{k}$.

Additionally, while each individual $(\widetilde{PgP})_{k-1}$ in a given subset of $S_{k-1}$ does not contribute to all subsets of $S_{k}$, collectively they do. To prove this we need to determine the contribution: $S_{k-1}^{(m)}\longrightarrow S_{k}^{(m')}$.

\begin{lemma}\label{L: equaldistro}
	Given the set $S_{l}=\left\{N: 5 \le N \le 4+P_{l}\#  \right\}$  containing at least one prospective prime pair with gap $g$, then for $k>l+4$,
	each subset $S_{k-1}^{(m)}\subset S_{k-1}$ generates a minimum of $(P_{k-2}-6)\cdot \mathring{n}_{k-3}^{g}\quad$ $(\widetilde{PgP})_{k}$ into each subset $S_{k}^{(m')}\subset S_{k}$.	
\end{lemma}
\begin{proof}
	
	From Lemma~\ref{L: distrofPgPinSj_2} each subset $S_{k-1}^{(m)}$ contains a minimum of  $P_{k-2}-4$ of  $(\widetilde{PgP})_{k-1}\in \left\{(\widetilde{PgP})_{k-1}\right\}_{(\widetilde{PgP})_{k-3}}$. These can be expressed as:
	
	\[
	(\widetilde{PgP})_{k-1}=(\widetilde{PgP})_{k-3}+m_{k-2}P_{k-3}\#+mP_{k-2}\#
	\]
	where they are distinguished by  $P_{k-2}-4$ distinct values of $m_{k-2}$.
	
	Then the contribution of these to $S_{k}^{(m')}$ is:
	
\[
	(\widetilde{PgP})_{k}=(\widetilde{PgP})_{k-3}+m_{k-2}P_{k-3}\#+mP_{k-2}\#+m'P_{k-1}\#\in S_{k}^{(m')}
\]
These are again distinguished by the $P_{k-2}-4$ distinct values of $m_{k-2}$ since we consider $m$ and $m'$ as constants, corresponding to two arbitrary subsets of $S_{k-1}$ and $S_{k}$ respectively.
	
Then we know from Lemma \ref{L: distroinSj+2} that for each value of $m_{k-2}$, each corresponding to a single $(\widetilde{PgP})_{k-2}\in S_{k-2}^{(m_{k-2})}$, that the disallowed subsets for each component of the resulting $(\widetilde{PgP})_{k}$ are separately distinct. But as discussed in the proof of Lemma \ref{L: distrofPgPinSj_2}, the disallowed subsets for the opposite components of two  $(\widetilde{PgP})_{k}$ may be the same. Therefore at most two of the $(\widetilde{PgP})_{k}$ may be disallowed in subset $S_{k}^{(m')}$, leaving a minimum of $P_{k-2}-6$ prospective prime pairs with gap $g$ in $S_{k}^{(m')}$ that are generated by such prospective prime pairs in $S_{k-1}^{(m)}$.
	
	Therefore given the existence of $S_{l}$ prescribed by the statement in the corollary, and given Theorem~\ref{T: noprotpp},we have:
	
	\[
	(P_{k-2}-6)\cdot \mathring{n}_{k-3}^{g}
	\]
	as the minimum contribution of $S_{k-1}^{(m)}$ to $S_{k}^{(m')}$.
	
\end{proof}

\begin{lemma}\label{L: uniformdistro}
	
	With respects to minimum distributions of prospective prime pairs with gap $g$, the contribution of $S_{k-1}^{(m)}$ to $S_{k}^{(m')}$ in the process of generating prospective prime pairs into $S_{k}$ from $S_{k-1}$ is asymtotically uniform across all subsets $m$ and $m'$.
\end{lemma}
\begin{proof}
	Lemma~\ref{L: equaldistro}	gives the minimum contributions of prospective prime pairs with gap $g$ from subset $S_{k-1}^{(m)}$ to subset $S_{k}^{(m')}$ as:
	\[
	(P_{k-2}-6)\cdot \mathring{n}_{k-3}^{g}
	\]
	Given that there are $P_{k-1}$ subsets in $S_{k-1}$ the total contribution from all subsets of $S_{k-1}$ is, at a minimum:
	\[
	P_{k-1}\cdot(P_{k-2}-6)\cdot \mathring{n}_{k-3}^{g}
	\]
	Then we know the minimum distribution of prospective prime pairs with gap $g$ from $S_{k-1}$ to each subset of $S_{k}$ is given by Theorem~\ref{T: distroPgP} as:
	
	\[
	\mathring{n}_{S_{k}^{(m)}}^{g}\ge \mathring{n}_{k-2}^{g}(P_{k-1}-4)
	\]	
	
	Taking the ratio of the minimum subset to subset contribution to the minimum contribution from set to subset gives:
	
	\begin{align}\notag
		\frac{P_{k-1}\cdot(P_{k-2}-6)\cdot \mathring{n}_{k-3}^{g}}{\mathring{n}_{k-2}^{g}(P_{k-1}-4)}=&\frac{P_{k-1}\cdot(P_{k-2}-6)}{(P_{k-2}-2)(P_{k-1}-4)}\\ \notag
		=&\frac{P_{k-2}-6}{P_{k-2}-6+4\left(1-\frac{P_{k-2}}{P_{k-1}}\right)+\frac{8}{P_{k-1}}}
	\end{align}
	The ratio is less than $1$ and clearly tends to $1$ for large $k$ proving the Lemma.
	
\end{proof}

\section{Prime pairs with gap g}
The foregoing results now allow the following theorem that proves the existence of actual prime pairs with gap $g$ given prospective prime pairs with gap $g$.

\begin{theorem}\label{T: numtwinprimes}
	Given a set $S_{r}=\left\{N: 5 \le N \le 4+P_{r}\#\right\}$ containing at least one prospective prime pair with gap $g$: $(\widetilde{PgP})_{r}$.
	Pick $l\ge r$ and define $P_{k}=P_{\boldsymbol{\pi}\left(\sqrt{P_{l}\#}\right)}$, then let $\mathring{n}_{P_{k}\rightarrow P_{k+1}^2}^{g}$ be the number of prime pairs with gap $g$ between $P_{k}$ and $P_{k+1}^2$ that are generated from $(\widetilde{PgP})_{r}$, then:
	\[
	\mathring{n}_{p_{k}\rightarrow P_{k+1}^{2}}^{g}\ge
\mathring{n}_{l}^{g}\cdot  \prod_{j=l}^{k-1}\frac{(P_{j}-4)  }{(P_{j}-2)}  \cdot \prod_{\substack{i=l\\ P_{i}|g }}^{k-1} \frac{(P_{i}-2)}{(P_{i}-1)}
	\]
	where, $\mathring{n}_{l}^{g}$ as in Theorem~\ref{T: noprotpp} is:
	\[
	\mathring{n}_{l}^{g}=\prod_{i=r+1}^{l}\left(P_{i}-2\right)\cdot\prod_{\substack{i=r+1  \\ P_{i}|g }}^{l}\frac{P_{i}-1}{P_{i}-2}
	\] 
	is the number of prospective prime pairs with gap $g$ in $S_{l}$ that are derived from each such prospective prime pair in $S_{r}$.
\end{theorem}

\begin{proof}
	Given $l$ and $P_{k}=P_{\boldsymbol{\pi}\left(\sqrt{P_{l}\#}\right)}$ consider the set of sequential natural numbers $S_{k}=\left\{5 \longrightarrow 4+P_{k}\#\right\}$. We will show that $S_{k}$ always contains prospective prime pairs, $(\widetilde{PgP})_{k} \in S_{k}$ prime to all $P \le P_{k}$ where $P_{k}< (\widetilde{PgP})_{k} < P_{k+1}^2$ and consequently those $(\widetilde{PgP})_{k}=(PgP)_{k}$ are actual prime pairs with gap $g$ and the number of such prime pairs meets the stated minimum.
	
	Note that while $P_{l}\#+1$ is the largest prospective prime number in $S_{l}=\left\{5 \longrightarrow 4+P_{l}\#\right\}$ in that it is prime to all $P \le P_{l}$,  it cannot be the square of a prime number.\footnote{Any prime number $>3$ has the form $6n\pm1$ and its square is then $36n^2\pm 12n +1$. Then equating $P_{l}\#+1$ to that square gives $6n^2\pm2n=\frac{P_{l}\#}{6}$. This cannot hold because the left side is even and the right is odd.}
	
	Therefore, with the definition of $P_{k}$ we have: 
	
	\[ 
	P_{k}^2 < P_{l}\# \longrightarrow P_{k}^2 \in S_{l}
	\]
	
	and given
	
	$P_{k+1}^2=\left(P_{\boldsymbol{\pi}\left(\sqrt{P_{l}\#}\right)+1}\right)^2$, we have:
	
	\[
	P_{l}\# < P_{k+1}^2 < P_{l+1}\#\longrightarrow  P_{k+1}^2 \in S_{l+1}\quad \& \quad P_{k+1}^2 \notin S_{l}
	\]
	
	Note that $P_{k+1}$ is the smallest prime number whose square is greater than $4+P_{l}\#$ and $P_{k}$ is the largest prime number whose square is less than $P_{l}\#$.  Therefore all prospective prime numbers and prospective prime pairs in $S_{l}$ are less than $P_{k+1}^2$. It remains to show that some $(\widetilde{PgP})_{l}$ are greater than $P_{k}$ and are prime to all $P \le P_{k}$ which means some $(\widetilde{PgP})_{l}=(\widetilde{PgP})_{k}$ and being less than $P_{k+1}^2$ are therefore actual prime pairs with gap $g$. In doing this we will show the inequality for $\mathring{n}_{P_{k}\rightarrow P_{k+1}^2}^{g}$ holds.
	
	To prove the theorem we must show there are some  $(\widetilde{PgP})_{l}=(\widetilde{PgP})_{k}$. Given that: 
	\[(\widetilde{PgP})_{l} \in S_{l}=S_{l+1}^{(0)}\subset S_{l+2}^{(0)}\subset \cdots \subset S_{k}^{(0)}
	\]
	This requires $m_{j}=0$ at each stage of: $(\widetilde{PgP})_{k}=(\widetilde{PgP})_{l}+\sum_{j=l+1}^{k}m_{j}P_{j-1}\#$. 
	
	We know, $S_{l+1}^{(0)}$ contains  a minimum number of prospective prime pairs with gap $g$, represented as $\min(\mathring{n}_{S_{l+1}^{(0)}}^{g})$ and given by Theorem~\ref{T: distroPgP}, which are prime to $P\le P_{l+1}$ and since $m_{l+1}=0$, $(\widetilde{PgP})_{l+1}=(\widetilde{PgP})_{l}$. 
	
	Then given $S_{l+2}^{(0)}=S_{l+1}$ we know again from  Theorem~\ref{T: distroPgP} that $S_{l+2}^{(0)}$ has a minimum number of prospective prime pairs with gap $g$ represented as $\min(\mathring{n}_{S_{l+2}^{(0)}}^{g})$ which are prime to $P\le P_{l+2}$. However all subsets of $S_{l+1}$ have contributed prospective prime pairs with gap $g$ to $S_{l+2}^{(0)}$ and we need to only consider those contributed by $S_{l+1}^{(0)}$. 
	
	Lemmas~\ref{L: equaldistro} and \ref{L: uniformdistro} showed that all subsets of $S_{l+1}$ contribute the same minimum number of prospective prime pairs to all subsets of $S_{l+2}$ and that the contributions remain uniform asymtotically for large $l$. Then the fraction of prospective prime pairs with gap $g$ in $S_{l+2}^{(0)}$ generated from $(\widetilde{PgP})_{l}=(\widetilde{PgP})_{l+1} \in S_{l+1}^{(0)}$ is therefore given by: 
	
	\[
	\frac{\min\left(\mathring{n}_{S_{l+1}^{(0)}}^{g}\right)}{\mathring{n}_{l+1}^{g}}\min\left(\mathring{n}_{S_{l+2}^{(0)}}^{g}\right)  =  \textrm{minimum number of}\quad (\widetilde{PgP})_{l+2}=(\widetilde{PgP})_{l}
	\]

	Then we have $\min\left(n_{S_{l+3}^{(0)}}^{g}\right)$ prospective prime pairs, $(\widetilde{PgP})_{l+3}\in S_{l+3}^{(0)}$ derived from all $(\widetilde{PgP})_{l+2}\in S_{l+2}$. The fraction of those derived from the set of $(\widetilde{PgP})_{l+2}=(\widetilde{PgP})_{l}\in S_{l+2}^{(0)}$ is:
	
	\begin{multline}\notag
	\frac{\min\left(\mathring{n}_{S_{l+1}^{(0)}}^{g}\right)}{\mathring{n}_{l+1}^{g}}\cdot\frac{\min\left(\mathring{n}_{S_{l+2}^{(0)}}^{g}\right)}{\mathring{n}_{l+2}^{g}}\cdot \min\left(\mathring{n}_{S_{l+3}^{(0)}}^{g}\right)\\ 
	= \textrm{minimum number of}\quad (\widetilde{PgP})_{l+3}=(\widetilde{PgP})_{l}
	\end{multline}
	
	Carrying this process forward up to the number of $(\widetilde{PgP})_{k}=(\widetilde{PgP})_{l}$, where then $P_{k}<(\widetilde{PgP})_{l}\le P_{k+1}^2$, gives:
	
	\begin{equation}\label{E: numtleqg}
		\mathring{n}_{p_{k}\rightarrow P_{k+1}^{2}}^{g} \ge \min\left(\mathring{n}_{S_{k}^{(0)}}^{g}\right) \prod_{j=l+1}^{k-1}\frac{\min\left(\mathring{n}_{S_{j}^{(0)}}^{g}\right)}{\mathring{n}_{j}^{g}}
	\end{equation}
	
	Expanding this using Theorem~\ref{T: distroPgP} 
	and Theorem~\ref{T: noprotpp} we get:
	
\begin{align}\label{E: abc}\notag
\mathring{n}_{p_{k}\rightarrow P_{k+1}^{2}}^{g}\ge& (P_{k-1}-4) \prod_{i=r+1}^{k-2}(P_{i}-2)\cdot \prod_{\substack{i=r+1\\ P_{i}|g }}^{k-2}\frac{(P_{i}-1)}{(P_{i}-2)}\cdot \\ \notag
&\cdot
\prod_{j=l+1}^{k-1}\frac{(P_{j-1}-4) \prod_{i=r+1}^{j-2}(P_{i}-2)\cdot \prod_{\substack{i=r+1\\ P_{i}|g }}^{j-2}\frac{(P_{i}-1)}{(P_{i}-2)}}{\prod_{i=r+1}^{j}(P_{i}-2)\cdot \prod_{\substack{i=r+1\\ P_{i}|g }}^{j}\frac{(P_{i}-1)}{(P_{i}-2)}}\\ \notag
=&
(P_{k-1}-4) \prod_{i=r+1}^{k-2}(P_{i}-2)\cdot \prod_{j=l+1}^{k-1}\frac{(P_{j-1}-4) \prod_{i=r+1}^{j-2}(P_{i}-2) }{\prod_{i=r+1}^{j}(P_{i}-2)}  \cdot\\ \notag
&\cdot \prod_{\substack{i=r+1\\ P_{i}|g }}^{k-2}\frac{(P_{i}-1)}{(P_{i}-2)}\cdot\prod_{\substack{j=l+1\\ P_{i}|g }}^{k-1} \frac{\prod_{\substack{i=r+1\\ P_{i}|g }}^{j-2}\frac{(P_{i}-1)}{(P_{i}-2)}}{\prod_{\substack{i=r+1\\ P_{i}|g }}^{j}\frac{(P_{i}-1)}{(P_{i}-2)}} \\ \notag
=&
(P_{k-1}-4) \prod_{i=r+1}^{l}(P_{i}-2)\cdot \prod_{i=l+1}^{k-2}(P_{i}-2)\cdot \prod_{j=l+1}^{k-1}\frac{(P_{j-1}-4)  }{{(P_{j}-2)(P_{j-1}-2)} }\\ \notag
&\cdot \prod_{\substack{i=r+1\\ P_{i}|g }}^{k-2}\frac{(P_{i}-1)}{(P_{i}-2)}\cdot \prod_{\substack{j=l+1\\ P_{j}|g }}^{k-1} \frac{(P_{j}-2)}{(P_{j}-1)}\frac{(P_{j-1}-2)}{(P_{j-1}-1)} \\ \notag
=&\prod_{i=r+1}^{l}(P_{i}-2)\cdot
\prod_{\substack{i=r+1\\ P_{i}|g }}^{l}\frac{(P_{i}-1)}{(P_{i}-2)}\cdot\prod_{j=l}^{k-1}\frac{(P_{j}-4)}{(P_{j}-2)}\cdot
\prod_{\substack{i=l\\ P_{i}|g }}^{k-1} \frac{(P_{i}-2)}{(P_{i}-1)}\\  
&=\mathring{n}_{l}^{g}\cdot  \prod_{j=l}^{k-1}\frac{(P_{j}-4)  }{(P_{j}-2)}  \cdot \prod_{\substack{i=l\\ P_{i}|g }}^{k-1} \frac{(P_{i}-2)}{(P_{i}-1)}
\end{align}	
	
This is clearly a possitive function and we want to show it is a monotonically increasing function with values greater than $1$. To do this we look at the case for $l\rightarrow l+1$ and $k\rightarrow k'=\boldsymbol{\pi}(\sqrt{P_{l+1}\#})$:

\begin{align}\notag
\mathring{n}_{p_{k'}\rightarrow P_{k'+1}^{2}}^{g}&\ge\mathring{n}_{l+1}^{g}\cdot  \prod_{j=l+1}^{k'-1}\frac{(P_{j}-4)  }{(P_{j}-2)}  \cdot \prod_{\substack{i=l+1\\ P_{i}|g }}^{k'-1} \frac{(P_{i}-2)}{(P_{i}-1)}\\ \notag
&=\prod_{r+1}^{l+1}(P_{i}-2)\prod_{\substack{i=r+1\\ P_{i}|g }}^{l+1} \frac{(P_{i}-2)}{(P_{i}-1)}\cdot  \prod_{j=l+1}^{k'-1}\frac{(P_{j}-4)  }{(P_{j}-2)}  \cdot \prod_{\substack{i=l+1\\ P_{i}|g }}^{k'-1} \frac{(P_{i}-2)}{(P_{i}-1)}\\ \notag
&=\mathring{n}_{l}^{g}\cdot(P_{l+1}-2)\cdot\left(\frac{P_{l+1}-2}{P_{l+1}-1}\right)_{P_{l+1}|g}\cdot \frac{(P_{l}-2)}{(P_{l}-4)}\cdot\prod_{i=l}^{k-1}\frac{(P_{i}-4)}{(P_{i}-2)}\cdot\\  \notag
&\cdot\prod_{i=k}^{k'-1}\frac{(P_{i}-4)}{(P_{i}-2)}\cdot\left(\frac{P_{l}-1}{P_{l}-2}\right)_{P_{l}|g}\cdot\prod_{\substack{i=l\\ P_{i}|g }}^{k-1} \frac{(P_{i}-2)}{(P_{i}-1)}\cdot\prod_{\substack{i=k\\ P_{i}|g }}^{k'-1} \frac{(P_{i}-2)}{(P_{i}-1)}\\  \notag
&=\mathring{n}_{p_{k}\rightarrow P_{k+1}^{2}}^{g}\cdot
(P_{l+1}-2)\cdot\left(\frac{P_{l+1}-2}{P_{l+1}-1}\right)_{P_{l+1}|g}\cdot \frac{(P_{l}-2)}{(P_{l}-4)}\cdot\\  \notag
&\cdot\prod_{i=k}^{k'-1}\frac{(P_{i}-4)}{(P_{i}-2)}\cdot\left(\frac{P_{l}-1}{P_{l}-2}\right)_{P_{l}|g}\cdot\prod_{\substack{i=k\\ P_{i}|g }}^{k'-1} \frac{(P_{i}-2)}{(P_{i}-1)}\\  \notag 
&=\mathring{n}_{p_{k}\rightarrow P_{k+1}^{2}}^{g}\cdot
(P_{l+1}-2)\cdot \frac{(P_{l}-2)}{(P_{l}-4)}\cdot\prod_{i=k}^{k'-1}\frac{(P_{i}-4)}{(P_{i}-2)}\cdot\\  \notag
&\cdot\left(\frac{P_{l+1}-2}{P_{l+1}-1}\right)_{P_{l+1}|g}\cdot\left(\frac{P_{l}-1}{P_{l}-2}\right)_{P_{l}|g}\cdot\prod_{\substack{i=k\\ P_{i}|g }}^{k'-1} \frac{(P_{i}-2)}{(P_{i}-1)}\\  \notag 
\end{align}

If we choose $l$ sufficiently large so that $P\ge P_{l}\rightarrow P\nmid g$, we can ignore the second line of products, giving:
\begin{align}\label{E: successiveterms}
	\mathring{n}_{p_{k'}\rightarrow P_{k'+1}^{2}}^{g}
	&\ge 
	\mathring{n}_{p_{k}\rightarrow P_{k+1}^{2}}^{g}\cdot
	(P_{l+1}-2)\cdot \frac{(P_{l}-2)}{(P_{l}-4)}\cdot\prod_{i=k}^{k'-1}\frac{(P_{i}-4)}{(P_{i}-2)}
\end{align}

Then the last product factor gives:

\begin{equation}\label{E: approxcase}
\prod_{i=k}^{k'-1}\frac{(P_{i}-4)}{(P_{i}-2)}=\prod_{i=k}^{k'-1}\left(1-\frac{2}{P_{i}-2}\right)\\ 
\ge 1-\frac{2(k-k')}{P_{k}}
\end{equation}

Then given $k'=\boldsymbol{\pi}(\sqrt{P_{l}\#})$, giving:

\[
k'\approx\frac{\sqrt{P_{l+1}\#}}{\ln{\sqrt{P_{l+1}\#}}}	=\frac{\sqrt{P_{l+1}}\sqrt{P_{l}\#}}{\ln{\sqrt{P_{l+1}}}+\ln{\sqrt{P_{l}\#}}}	
\]

Ignoring $\ln{\sqrt{P_{l+1}}}$ relative to $\ln{\sqrt{P_{l}\#}}$ and noting that $k\approx \frac{\sqrt{P_{l}\#}}{\ln{\sqrt{P_{l}\#}}}$, gives:
\[
k'\approx \sqrt{P_{l+1}}\cdot k
\]
Using this in (\ref{E: approxcase}) gives:

\begin{align}\label{E: approxfinal}
	\prod_{i=k}^{k'-1}\frac{(P_{i}-4)}{(P_{i}-2)} 
	&\ge 1-\frac{2\sqrt{P_{l+1}}}{\ln{P_{k}}}\ge   1-\frac{2\sqrt{P_{l+1}}}{\ln{\sqrt{P_{l+1}\#}}}
\end{align}

Therefore $\prod_{i=k}^{k'-1}\frac{(P_{i}-4)}{(P_{i}-2)}$, while remaining $<1$ is a monotonically increasing function assymtotically approaching $1$.  The approximation (\ref{E: approxfinal}) is conservative:\footnote{The approximation used in (\ref{E: approxfinal}) allows negative values for small $l$, but is positive for $l\ge 8$, while the term being approximated clearly always has a positive value.}, and  using it for the last term in (\ref{E: successiveterms}) gives for example:

\[
l= 9:\qquad \mathring{n}_{p_{k'}\rightarrow P_{k'+1}^{2}}^{g}
\ge 
1.4\cdot \mathring{n}_{p_{k}\rightarrow P_{k+1}^{2}}^{g}
\]
\[
l= 10:\qquad \mathring{n}_{p_{k'}\rightarrow P_{k'+1}^{2}}^{g}
\ge 
4.5 \cdot\mathring{n}_{p_{k}\rightarrow P_{k+1}^{2}}^{g}
\] 
\[
l= 15:\qquad \mathring{n}_{p_{k'}\rightarrow P_{k'+1}^{2}}^{g}
\ge 
18.7\cdot\mathring{n}_{p_{k}\rightarrow P_{k+1}^{2}}^{g}
\]

\end{proof}

Given Theorem~\ref{T: numtwinprimes} we can prove the following theorem:

\begin{theorem}\label{T: main}
	Given a set $S_{r}=\left\{N: 5 \le N \le 4+P_{r}\#\right\}$ containing at least one prospective prime pair with gap $g$, then given any number $M$ there is always a prime pair with gap $g$ greater than $M$.	
\end{theorem}	

\begin{proof}
	Pick integer $l>r$ so that $P_{k}=P_{\boldsymbol{\pi}\left(\sqrt{P_{l}\#}\right)}>M$.Then we know from Theorem~(\ref{T: numtwinprimes}) that there is always a prime pair with gap $g$  greater than $P_{k}$.	
\end{proof}	

\section{Prime gaps for which de Polignac's conjecture holds}

Given Theorem~\ref{T: numtwinprimes} we need only show the existence of a set $S_{k}$ containing a pair of consecutive prospective prime numbers with a secific gap $g$ to prove de Polignac's conjecture holds for that gap. 

\begin{lemma}\label{L: primegaps}
	Given any prime number $P_{k}>3$, then $P_{k}$ and $P_{k+1}$ are consecutive prospective prime numbers in $S_{k-1}$. 
\end{lemma}
\begin{proof}
	Consider the set $S_{k-1}=\left\{N: 5 \le N \le 4+P_{k-1}\#\right\}$ and its subset of prospective prime numbers, $\widetilde{\mathbb{P}}_{k-1}$.
	
	Then we know that all prospective prime numbers in $\widetilde{\mathbb{P}}_{k-1}$ that are less than $P_{k}^2$ are actual prime numbers.
	For $P_{k}>3$, we have:
	\[
	P_{k+1}<P_{k-1}\#+4\quad \textrm{and consequently}\quad P_{k},P_{k+1}  \in \widetilde{\mathbb{P}}_{k-1}
	\] 
	and because $P_{k},P_{k+1}<P_{k}^2$, any prospective prime number between them must also be an actual prime number. But $P_{k}$ and $P_{k+1}$ are consecutive prime numbers, so there can be no prospective prime numbers between them and they are consecutive prospective prime numbers as well as consecutive actual prime numbers in $S_{k-1}$.  
\end{proof}

The following theorem follows directly from Theorem~\ref{T: main} together with Lemma~\ref{L: primegaps}

\begin{theorem}\label{T: depolignacPkpm2}
	For all $P_{k}>3 $ there exists infinitely many consecutive prime pairs with gaps $g=P_{k+1}-P_{k}$.
\end{theorem}

Now consider the gaps between subsets, where we use the following definitions:

\begin{definition}
	
	\[
	\widetilde{P}_{\{k\}}^{(m)<}:=\min \left\{ \widetilde{P}\in S_{k}^{(m)} \right\}
	\]
	\[
	\widetilde{P}_{\{k\}}^{(m)>}:=\max \left\{ \widetilde{P}\in S_{k}^{(m)} \right\}
	\]
	
\end{definition}

Then the subset gap is defined as:

\begin{definition}
	\[
	g_{\Delta_{SS_{k}}}:=\widetilde{P}_{\{k\}}^{(m)<}-\widetilde{P}_{\{k\}}^{(m-1)>}
	\]
\end{definition}

\begin{lemma}\label{L: subsetgaps}
	Given set $S_{k}=\left\{N: 5 \le N \le 4+P_{k}\#\right\}$ and its $P_{k}$ subsets $S_{k}^{(m)}=\left\{N: 5+m\cdot P_{k-1}\# \le N \le 5+(m+1)\cdot P_{k-1}\#\right\}$ with $P_{k}-1$ associated gaps, $g_{\Delta_{SS_{k}}}$, then:
	\begin{equation}\notag
		g_{\Delta_{SS_{k}}}=
		\begin{cases}
			P_{k}-1 \quad \textrm{for $P_{k}-2$ gaps}\\
			P_{k}+1 \quad \textrm{for one gap}
		\end{cases}
	\end{equation}
\end{lemma}
\begin{proof}
	The smallest prospective prime in $S_{k-1}$ is $P_{k}$ and the largest two prospective primes in $S_{k-1}$ are $\widetilde{P}_{k-1}^{\pm}:=P_{k-1}\#\pm 1$.
	
	For $S_{k-1}\rightarrow S_{k}$ we use (\ref{E: nextproprime}) subject to the supplementary condition (\ref{E: defmhat}) to generate prospective primes in $S_{k}$. Note that given $P_{k}\in \widetilde{\mathbb{P}}_{k-1}$, then for $\widetilde{P}_{k}=P_{k}+mP_{k-1}\#$ all values of $m$ except $m=0$ are allowed, making $P_{k+1}$ the least prospective prime in the zeroth subset of  $\widetilde{\mathbb{P}}_{k}$, and making $P_{k}+mP_{k-1}\#$ the least prospective prime in all other subsets of $\widetilde{\mathbb{P}}_{k}$. Therefore we have:
	
	\begin{equation}\label{E: firstpprime}
		\widetilde{P}_{\{k\}}^{(m)<}=
		\begin{cases}
			P_{k+1}&\textrm{for}\quad m=0   \\
			P_{k}+m P_{k-1}\#&\textrm{for}\quad 1\le m\le P_{k}-1
		\end{cases}
	\end{equation}	
	and
	\begin{equation}\label{E: lastpprime}
		\widetilde{P}_{\{k\}}^{(m)>}=
		\begin{cases}
			(m +1)P_{k-1}\# +1 &\textrm{if}\quad  m \ne	\widehat{m}^+\\
			(m +1)P_{k-1}\# -1 &\textrm{if}\quad  m= 	\widehat{m}^+
		\end{cases}	
	\end{equation}	
	$\widehat{m}^+$ represents the disallowed subset for $\widetilde{P}_{\{k\}}=\widetilde{P}_{k-1}^{+}+mP_{k-1}\#$, which however is allowed for $\widetilde{P}_{k-1}^{-}=\widetilde{P}_{k-1}^{+}-2$, where:
	\[
	\widehat{m}^+=\frac{\alpha P_{k}-\widetilde{P}_{k-1}^{+}\bmod{P_{k}}}{P_{k-1}\#\bmod{P_{k}}}          \\ 
	\]
	
	In regards to (\ref{E: lastpprime}) note that $\widehat{m}^+\ne P_{k}-1$ because using the maximum value is always allowed for $P_{k-1}^{\pm}$ where using it in (\ref{E: nextproprime}) gives:
	
	\[
	\widetilde{P}_{k-1}^{\pm}+(P_{k}-1)P_{k-1}\# =P_{k-1}\#\pm 1 +(P_{k}-1)P_{k-1}\# =P_{k}\#\pm 1 =P_{k}^{\pm}
	\]
	
	Therefore, $\widehat{m}^+$ associated with $\widetilde{P}_{k-1}^{+}$ can only have a value in the range $0$ to $P_{k}-2$ associated with the greatest prospective prime in each subset of $\widetilde{\mathbb{P}}_{k}$. This leaves one subset of $\widetilde{\mathbb{P}}_{k}$, namely $\mathbb{P}_{k}^{(\widehat{m}^+)}$, $0\le \widehat{m}^+\le P_{k}-2$ where $\widehat{m}^+ P_{k-1}\#-1$ is the greatest prospective prime and where $mP_{k-1}\#+1$ is the greatest prospective prime in the remainder of the subsets.  Therefore, there are $P_{k}-2$ cases where:
	
	\[
	g_{\Delta_{SS_{k}}}=\widetilde{P}_{\{k\}}^{(m)<}-\widetilde{P}_{\{k\}}^{(m-1)>}= (P_{k}+mP_{k-1}\#)-( mP_{k-1}\# +1)=P_{k}-1
	\] 
	$1\le m\le P_{k}-1$, and $m-1 \ne \widehat{m}^+$;
	
	and one case where:
	\[
	g_{\Delta_{SS_{k}}}=\widetilde{P}_{\{k\}}^{(\widehat{m}^+)<}-\widetilde{P}_{\{k\}}^{(\widehat{m}^+ -1)>}= (P_{k}+\widehat{m}^+P_{k-1}\#)-( \widehat{m}^+P_{k-1}\# -1)=P_{k}+1
	\] 
\end{proof}

\begin{corollary}\label{C: twogapvalues} 
	Every set $S_{k}$ has at least $P_{k}-2$ prospective prime pairs with gap $g=P_{k}-1$ and at least one prospective prime pair with gap $g=P_{k}+1$	
\end{corollary}
\begin{proof}
	This follows directly from Lemma~\ref{L: subsetgaps} recognizing that gaps between subsets are gaps between prospective prime pairs. The "at least" follows because internal to subsets there are prime pairs with gaps that may be the same or may differ from the subset gaps.	
\end{proof}

The following theorem follows directly from Theorem~\ref{T: main} and Corollary~\ref{C: twogapvalues}

\begin{theorem}\label{T: depolignacPkpm1}
	For all $P_{k}$ there exists infinitely many consecutive prime pairs with gaps $g=P_{k} \pm 1$.
\end{theorem}




\begin{thebibliography}{99}


\bibitem{BFI1}
Bombieri, E.; Friedlander, J.B.; and Iwaniec, H.; "Primes in arithmetic progressions to large moduli", Acta. Math., 156(3-4):203-251, 1986

\bibitem{BFI2}
Bombieri, E.; Friedlander, J.B.; and Iwaniec, H.; "Primes in arithmetic progressions to large moduli II", Math. Ann.,277(3):361-393, 1987 

\bibitem{BFI3}
Bombieri, E.; Friedlander, J.B.; and Iwaniec, H.; "Primes in arithmetic progressions to large moduli III", J. Am Math. Soc., 2(2):215-224, 1989

\bibitem{GPY}
Goldston D.A.; Pinz, J.; Yildirim, C.Y.; "Primes in Tuples I"; Ann. of Math.(2), 170(2):819-862, 2009

\bibitem{VB}
Brun, V., "Le crible d'Eratosth\`{e}nes et le th\'{e}oreme d Goldbach", C.R., Acad. Sci. Paris, 168 (1919) 544-546;


\bibitem{dP}
de Polignac, A. ``Six propositions arithmologiques déduites de crible d'Ératosthène.'' Nouv. Ann. Math. 8, 423-429, 1849.

\bibitem{LD}
Dickson, L. E. \underline{History of the Theory of Numbers, Vol. 1: Divisibility and }
\underline{Primality},424, New York: Dover, 2005.

\bibitem{HL}
Hardy, G.H. and Littlewood, J.E. "Some problemsss 'Partitio numerorum'; III: on the expression off a number as a sum of primes." Acta. Math., 44:1-70, 15 February 1922

\bibitem{JM1}
Maynard, J.A., "On the twin prime conjecture", 
arXiv:1910.14674v1 [math.NT], 29 Oct 2019

\bibitem {JM2} 
Maynard, J.A.; "Small gaps between primes", Ann. of Math.,(2), 181(1):383-413, 2015

\bibitem{poly1}
Polymath, D.H.J.; "Variants of the Selberg sieve, and bounded intervals containing many primes", Res. Math. Sci., 1:art. 12,83, 2014


\bibitem{poly2}
Polymath, D.H.J.; Castryck, W.; Fouvry, E.; Harcos, G.; Kowalski, E.; Michel, P.; Nelson, P.; Paldi, E.; Pintz, J,; Sutherland, A.V.; Tao,T.; Xie, X.-F.; "New equidistribution estimates of Zhang type." Algebra Number Theory, 8(9):2067-2199, 2014

\bibitem{AS}
Selberg, A. "The general sieve method and its place in prime number theory." Proc. Internat. Congr. Math., Cambridge Mass. 1950, vol 1,286-291

\bibitem{JS}
J. Sellers,"Distribution of twin primes in repeating sequences of prime factors",
arXiv:2108.00288[math.GM] 31 Jul 2021
 
\bibitem{YZ} 
Zhang, Y. "Bounded gaps between primes", Ann. of Math. (2),179(3):1121-1174,2013
 
\end{thebibliography}


\end{document}